\documentclass{amsart}
\usepackage[utf8]{inputenc}
\usepackage{amsfonts,latexsym,amssymb,amsmath,amsthm,color}
\usepackage{enumerate}
\usepackage{captdef}
\usepackage{enumerate,graphicx,graphpap}
\usepackage[matrix]{xy}
\usepackage{cases}
\usepackage{mathrsfs}
\usepackage{enumitem}
\usepackage{multicol}
\usepackage{lscape}
\usepackage{caption}
\usepackage{tikz}
\usepackage{multirow}
\usepackage{makecell}

\def\N{\mathbb N}
\def\Z{\mathbb Z}

\def\R{\mathbb R}
\def\C{\mathbb C}
\def\a{\alpha}

\def\d{\partial}

\theoremstyle{plain}
\newtheorem{teor}{Theorem}[section]
\newtheorem{lema}[teor]{Lemma}

\newtheorem{prop}[teor]{Proposition}

\theoremstyle{definition}

\newtheorem{eje}{Example}[section]
\newtheorem{nota}[teor]{Remark}

\title[On characterizations of Appell and Sheffer sequences and examples]{Appell and Sheffer sequences: on their characterizations through functionals and examples}
	
\author{Sergio A. Carrillo}
\author{Miguel Hurtado}

\address{Programa de matem\'{a}ticas, Universidad Sergio Arboleda, Calle 74, $\#$ 14-14, Bogot\'{a}, Colombia.}\email{sergio.carrillo@usa.edu.co, miguel.hurtado01@usa.edu.co}

\thanks{First author is supported by Ministerio de Econom\'{i}a y Competitividad from Spain, under the Project ``M\'{e}todos asint\'{o}ticos, algebraicos y geom\'{e}tricos en foliaciones singulares y sistemas din\'{a}micos" (Ref.: PID2019-105621GB-I00) and Univ. Sergio Arboleda project IN.BG.086.20.002.}



\begin{document}
\maketitle

\begin{abstract} The aim of this paper is to present a new simple recurrence for Appell and Sheffer sequences in terms of the linear functional that defines them, and to explain how this is equivalent to several well-known characterizations appearing in the literature. We also give several examples, including integral representations of the inverse operators associated to Bernoulli and Euler polynomials, and a new integral representation of the re-scaled Hermite $d$-orthogonal polynomials generalizing the Weierstrass operator related to the Hermite polynomials.
\end{abstract}

\section{Introduction}

A remarkable class of polynomials are the Appell sequences having applications in Number theory, Probability, and the theory of functions. They vastly generalize monomials arising naturally from Taylor's formula with integral rest, see \cite[Chapitre VI]{Bourbaki} \cite[Appendix A]{Candel} and they include famous polynomial sequences. For instance, the Bernoulli polynomials useful in numerical integration and asymptotic analysis --Euler-Maclaurin formula \cite{Olver}--; or the Euler polynomials which lead to Euler-Boole formula \cite{Borwein}. On the other hand, a wide class of Appell sequences provides special values of transcendental functions, as recently proved in \cite{Navas2018}, extending the well-known case of the Bernoulli polynomials as the values at negative integers of the Hurwitz zeta function. 

Appell sequences are a subclass of Sheffer sequences --\textit{type zero} polynomials-- which were treated by I. M. Sheffer as solutions of families of differential and difference equations \cite{Sheffer1939}. Later on, their study, leaded by Rota and Roman, developed in what it is known today as Umbral Calculus \cite{Roman,RomanRota,Rota}. In contrast, a more recent approach has been made using matrix and determinantal representations, see e.g., \cite{Aceto2015,Aceto2017,CosLongo2010,CosLongo2014,Wang2014,Yang}. Also, current research has focussed on special sequences \cite{Drissi} and other alternative descriptions of the theory, for instance, through random variables \cite{Ta2015,AdellLekuona17}.

Our goal in this work is to obtain a new simple recursion for Appell sequences (Theorem \ref{Thm New}), their expansions in terms of an arbitrary delta operator (Proposition \ref{Prop. 2}) -as presented in \cite{AdellLekuona} for difference operators-, and several examples. Our exposition is based on Umbral Calculus, briefly recalled in Sections \ref{Sec:2} and \ref{Sec:3}. Using this tool, we obtain in Section \ref{Sec:3} a new recurrence for Sheffer sequences (Theorem \ref{Coro Last}) using left-inverses of a delta operator and deduce Theorem \ref{Thm New} as a particular case. In addition, we clarify how the characterization of Sheffer and Appell sequences through a linear functional and a linear operator are naturally equivalent (Theorem \ref{Thm2}) --a general fact used systematically in specific cases, see, e.g., \cite{Borwein}, \cite[Theorems 2.2-2.3]{Drissi}, \cite{Hassen}, \cite{Tempesta}.-- Our examples are presented in Section \ref{Sec:5}, in particular, we generalize the Weierstrass operator for classical Hermite polynomials by using Ecalle's accelerator functions to obtain an integral representation of the re-scaled Hermite $d$-orthogonal polynomials \cite{Douak96} (Proposition \ref{Prop. Hnd}). We also include the integral representations of the inverse operators associated with the Bernoulli and Euler polynomials (Propositions \ref{Prop. Inv Ber} and \ref{Prop Inv Euler}). These formulas are closely related to their moment expansions \cite{Sun2007,Ta2015}, but here we deduce them in a direct elementary way. Finally, we collect in Table \ref{Table} multiple examples of Appell sequences scattered in the literature, with their respective linear functional and characterization.

\section{Preliminaries on Sheffer sequences}\label{Sec:2}

We briefly recall some characterizations of Sheffer sequences and set the notations used along the paper. Our summary is based on the expositions \cite{Rota,Roman} of Umbral Calculus whose cornerstone is the twofold identification of formal power series in one variable $\C[[t]]$ as the linear functionals, as well as the shift-invariant linear operators of the ring of univariate polynomials $\C[x]$. 

Let $\d=\d_x$ be usual differentiation and $T_a:\C[x]\to\C[x]$,  $T_a(p)(x)=p(x+a)$ the shift-operator indexed by $a\in\C$. A linear operator $\mathfrak{Q}:\C[x]\to \C[x]$ is \textit{shift-invariant} if $\mathfrak{Q}\circ T_a=T_a\circ \mathfrak{Q}$, for all $a\in\C$. The set $\Sigma$ of these operators acquires a commutative ring structure via the isomorphism \begin{equation}\label{Eq. Ring iso}
\iota_{\d}:\C[[t]]\to\Sigma,\qquad  \overline{B}(t)=\sum_{n=0}^\infty \frac{\hat{b}_n}{n!}t^n\longmapsto \mathfrak{Q}=\overline{B}(\d)=\sum_{n=0}^\infty \frac{\hat{b}_n}{n!}\d^n,
\end{equation} where the product of formal power series corresponds to composition of operators. In fact,  $\hat{b}_n=\mathfrak{Q}(x^n)(0)$. Extending  $\mathfrak{Q}$ to $\C[x][[t]]$ by  $\mathfrak{Q}\left(\sum_{n=0}^\infty p_n t^n \right)=\sum_{n=0}^\infty \mathfrak{Q}(p_n)t^n$ we recover $\overline{B}(t)$ using the value of $\mathfrak{Q}$ at the exponential. Indeed, $\mathfrak{Q}(e^{xt})(x,t)=\overline{B}(t)e^{xt}$ and therefore $\iota_\d^{-1}(\mathfrak{Q})=\mathfrak{Q}(e^{xt})(0,t)$.

An operator $\mathfrak{Q}\in\Sigma$ is called a \textit{delta operator} if $\mathfrak{Q}(x)$ is a non-zero constant. In this case $\mathfrak{Q}(a)=0$, for all $a\in\C$ and $\text{deg}(\mathfrak{Q}(p))=\text{deg}(p)-1.$ Thus the series $\overline{B}(t)$ starts at $n=1$ with $\hat{b}_1=\mathfrak{Q}(x)\neq 0$ and it admits a compositional inverse $B(t)=\sum_{n=1}^\infty \frac{b_n}{n!}t^n$. The sequence of polynomials $\{q_n(x)\}_{n\geq0}$ induced through the expansion (\ref{Eq. B(t)}) below are called \textit{basic polynomials} of $\mathfrak{Q}$, \begin{equation}\label{Eq. B(t)} e^{xB(t)}=\sum_{n=0}^\infty \frac{q_n(x)}{n!}t^n,\quad  \text{ and they satisfy }  q_n(x+x_0)=\sum_{j=0}^n \binom{n}{j} q_j(x)q_{n-j}(x_0). 
\end{equation} Thus, they are of \textit{binomial type}. Moreover, they are characterized by the properties \begin{equation}\label{Eq. Basic poly Q}
q_0(x)=1,\quad q_n(0)=0, \text{ and} \quad \mathfrak{Q}(q_n)(x)=nq_{n-1}(x), \text{ for all } n\geq1.
\end{equation}

A delta operator $\mathfrak{Q}$ also induces a ring isomorphism $\iota_{\mathfrak{Q}}:\C[[t]]\to\Sigma$ by $A(t)\to A(\mathfrak{Q})$, and if $\mathfrak{S}=A(\mathfrak{Q})=\sum_{n=0}^\infty \frac{a_n}{n!}\mathfrak{Q}^n$, then $a_n=\mathfrak{S}(q_n)(0)$. Therefore $
\mathfrak{S}(e^{xB(t)})(x,t)=A(t) e^{xB(t)}$ and we recover  $\iota_{\mathfrak{Q}}^{-1}(\mathfrak{S})=\mathfrak{S}(e^{xB(t)})(0,t)$. Also, note that $\mathfrak{S}$ is invertible if and only if $a_0=\mathfrak{S}(1)\neq0$, and $\mathfrak{S}^{-1}=1/A(\mathfrak{Q})$, where $1/A(t)$ is the reciprocal of $A(t)$. 

We say $\{s_n(x)\}_{n\geq0}$ is a $\mathfrak{Q}$--\textit{Sheffer sequence} for the delta operator $\mathfrak{Q}$  if $s_0\neq0$ is constant and \begin{equation}\label{Eq. Qsn} \mathfrak{Q}(s_n)(x)=n s_{n-1}(x), \text{ for all }n\geq 1.
\end{equation} These sequences admit several characterizations. First, there is an invertible $\mathfrak{S}\in\Sigma$ satisfying \begin{equation}\label{Eq. Ssn=qn}
\mathfrak{S}(s_n)=q_n,\text{ for all } n\geq0, 
\end{equation}  Second, its exponential generating series has the form \begin{equation}\label{Eq. Exp sn}
\sum_{n=0}^\infty \frac{s_n(x)}{n!}t^n=\frac{e^{xB(t)}}{A(t)},
\end{equation} where $A(t)\in\C[[t]]$ has a reciprocal, i.e., $A(0)\neq 0$. Third, the sequence satisfies \begin{equation}\label{Eq. Binom sn}
s_{n}(x+x_0)=\sum _{k=0}^{n} \binom{n}{k} s_{k}(x)q_{n-k}(x_0), \text{ and } s_n(x)=\sum_{k=0}^n \binom{n}{k} s_k(0) q_{n-k}(x), \text{ for all } n\geq0.
\end{equation} Conditions (\ref{Eq. Qsn}), (\ref{Eq. Ssn=qn}), (\ref{Eq. Exp sn}) and both equations in (\ref{Eq. Binom sn}) are equivalent to each other as can be checked. The relevant relations are $\mathfrak{S}=\iota_{\mathfrak{Q}}(A)$ and $1/A(t)=\sum_{n=0}^\infty s_n(0) t^n/n!$. We highlight that $\mathfrak{S}$ is uniquely associated to $\{s_n\}_{n\geq0}$ which will be referred as the $\mathfrak{Q}$--Sheffer sequence \textit{relative} to $\mathfrak{S}$ (\textit{$(\mathfrak{Q},\mathfrak{S})$--Sheffer} for short). Finally, after repeated application of $\mathfrak{Q}$ to (\ref{Eq. Ssn=qn}) followed by evaluating at $x=0$, we find that \begin{equation}\label{Eq. SQm}
\mathfrak{S}(\mathfrak{Q}^{ m}(s_n))(0)=n!\delta_{n,m},\quad \text{ for all } n,m\geq0,
\end{equation} where $\delta_{n,m}$ is the Kronecker delta.

\section{The use of linear functionals}\label{Sec:3}

Another characterization of Sheffer sequences is available through functionals of $\C[x]$. It is based on the identification of the dual space $\C[x]^\ast$ with $\Sigma$ --and thus with $\C[[t]]$  via (\ref{Eq. Ring iso})--.

\begin{lema}\label{Prop. LL} We have the linear isomorphism $\mathfrak{j}:\C[x]^\ast\to \Sigma$ given by 
	\begin{equation}\label{Eq. L LL}
	L\longmapsto \mathfrak{L}(p)(x):=L(T_{x}(p)),\quad \text{ and having as inverse }\quad \mathfrak{L}\longmapsto L(p)=\mathfrak{L}(p)(0),\quad p\in\C[x].
	\end{equation}
\end{lema}

\begin{proof} Given $L\in\C[x]^\ast$ the map $\mathfrak{L}=\mathfrak{j}(L)$ is clearly linear. It is also shift-invariant since  $$\mathfrak{L}(T_{a}(p))(x)=L(T_{x}(T_a(p)))=L(T_{x+a}(p))=\mathfrak{L}(p)(x+a)=T_a(\mathfrak{L}(p))(x),\quad \text{ for all }a\in\C.$$ Conversely, if $\mathfrak{L}\in\Sigma$ and $L(p)=\mathfrak{L}(p)(0)$, then  $L\in\C[x]^\ast$ and $
	L(T_{x_0}(p))=\mathfrak{L}(T_{x_0}(p))(0)$$=(T_{x_0}\circ \mathfrak{L})(p)(0)=\mathfrak{L}(p)(x_0)$, for all  $x_0\in\C$. Thus the maps in (\ref{Eq. L LL}) are inverses one of each other. 
\end{proof}

\begin{nota} We can codify $L\in\C[x]^\ast$ by the values $L_n:=L(x^n)$ known as the \textit{moments} of $L$. Extending $L$ to $\C[x][[t]]$ by $L(\sum_{n=0}^\infty p_n(x)t^n)=\sum_{n=0}^\infty L(p_n)t^n$ we have the relation $\mathfrak{L}(e^{xB(t)})(x_0,t)=L(T_{x_0}(e^{xB(t)}))=L(e^{(x+x_0)B(t)})=e^{x_0 B(t)}L(e^{xB(t)})$. In particular, the \textit{indicator series} of $L$ \begin{equation}
	L(e^{xt})=\sum_{n=0}^\infty L_n\frac{t^n}{n!},\quad \text{ satisfies }\quad  \mathfrak{L}(e^{xt})(x,t)=L(e^{xt})e^{xt}. 
	\end{equation}
\end{nota}

Now we are in position to give two equivalent ways to characterize Sheffer sequences using functionals.

\begin{teor}\label{Thm2} Given $S\in\C[x]^\ast$, there is a unique $\mathfrak{Q}$-Sheffer sequence $\{s_n\}_{n\geq0}$   satisfying \begin{equation}\label{Eq. Ssn=0}
	S(s_0)=1\quad \text{ and }\quad S(s_n)=0, \text{ for all  } n\geq 1,
	\end{equation} or equivalently, 
	\begin{equation}\label{Eq. STsn=qn}
	S(T_{x_0}(s_n))=q_n(x_0),\quad \text{ for all } n\geq 0 \text{ and } x_0\in\C.\end{equation} Indeed, $\{s_n\}_{n\geq0}$ is the  $(\mathfrak{Q},\mathfrak{S})$--Sheffer sequence, $\mathfrak{S}=\mathfrak{j}(S)$, having exponential generating series $e^{xB(t)}/S(e^{xB(t)})$. Conversely, given a $\mathfrak{Q}$-Sheffer sequence $\{s_n\}_{n\geq0}$, there is a unique $S\in\C[x]^\ast$ such that (\ref{Eq. Ssn=0}) holds.
\end{teor}

\begin{proof} First note that (\ref{Eq. Ssn=0}) is simply (\ref{Eq. STsn=qn}) for $x_0=0$ as (\ref{Eq. Basic poly Q}) shows. Conversely, if (\ref{Eq. Ssn=0}) holds, then (\ref{Eq. STsn=qn}) follows by applying $S$ to the first equation in  (\ref{Eq. Binom sn}). Furthermore, if $\mathfrak{S}=\mathfrak{j}(S)$, then equation (\ref{Eq. Ssn=qn}) is equivalent to (\ref{Eq. STsn=qn}) since $S(T_{x_0}(s_n))=\mathfrak{S}(s_n)(x_0)$, thanks to Lemma \ref{Prop. LL}. Also, applying $S$ to equation (\ref{Eq. Exp sn}) we find $A(t)=S(e^{xB(t)})$. 
	
	Finally, given the  $(\mathfrak{Q},\mathfrak{S})$-Sheffer sequence $\{s_n\}_{n\geq0}$, if $S'\in\C[x]^\ast$ satisfies (\ref{Eq. Ssn=0}), then $\mathfrak{S}'=\mathfrak{j}(S')\in\Sigma$ is invertible and satisfies $\mathfrak{S}'(s_n)=q_n$ for all $n$. Since $\mathfrak{S}$ is characterized by this condition, then $\mathfrak{S}'=\mathfrak{S}$ and $S'=\mathfrak{j}^{-1}(\mathfrak{S})=S$ is also uniquely determined. 
\end{proof}

The previous theorem shows that it is equivalent to have a $\mathfrak{Q}$--Sheffer sequence $\{s_n\}_{n\geq0}$, an invertible operator $\mathfrak{S}\in\Sigma$ or a functional $S\in\C[x]^\ast$ such that $S(1)\neq 0$. Thus we can refer to $\{s_n\}_{n\geq0}$ as the \textit{$(\mathfrak{Q},\mathfrak{S},S)$--Sheffer sequence}, where $\mathfrak{S}=\mathfrak{j}(S)$.

\begin{nota}\label{Nota kfold} The $k$-fold iteration $\mathfrak{S}^k$ of an invertible operator $\mathfrak{S}\in\Sigma$ produces the $(\mathfrak{Q},\mathfrak{S}^k,S^k)$--Sheffer sequence $\{s_n^{(k)}\}_{n\geq0}$,  $S^k:=\mathfrak{j}^{-1}(\mathfrak{S}^k)$, having exponential generating series $e^{xB(t)}/S(e^{xB(t)})^k$. According to  (\ref{Eq. Ssn=qn}) we find  $s_n^{(k)}=\mathfrak{S}^{\circ (-k)}(q_n)=\mathfrak{S}\circ\mathfrak{S}^{\circ (-k-1)}(q_n)=\mathfrak{S}(s_n^{(k+1)})$ holding for all $k,n\in\N$. We also highlight that if $S$ admits the representation $$S(p)=\int_I p(s)w(s)ds, \quad \text{ and }\quad  \mathfrak{S}(p)(x)=\int_I p(x+s)w(s)ds,$$ where $I\subseteq \R$ is an interval and $w:I\to \C$ is such that $S_n$ are all finite, then $$S^{k}(p)=\int_{I^k} p(s_1+\cdots+s_k)w(s_1)\cdots w(s_k)d\boldsymbol{s},\quad  \mathfrak{S}^{ k}(p)(x)=\int_{I^k} p(x+s_1+\cdots+s_k)w(s_1)\cdots w(s_k)d\boldsymbol{s},$$ where $d\boldsymbol{s}=ds_1\cdots ds_k$ and the integration is taken over the $k$th Cartesian product $I^k\subseteq \R^k$.
\end{nota}

A $\mathfrak{Q}$--Sheffer sequence can be calculate though its generating series, or using determinants, see \cite{CosLongo2010,CosLongo2014,Wang2014}. Here we present a new recursion formula using left-inverses for $\mathfrak{Q}$. 

\begin{teor}\label{Coro Last} Let $\mathfrak{Q}$ be a delta operator and $\overline{B}=\iota_\d^{-1}(\mathfrak{Q})$. Then each $x_0\in\C$ defines  $$\mathfrak{Q}^{-1}_{x_0}(p):=\frac{\d}{\overline{B}(\d)}\left(\int_{x_0}^x p(s)ds\right),\text{ which is a left-inverse for } \mathfrak{Q}.$$ Here the integral is taken over the line segment from $x_0$ to $x$. Through it, the $(\mathfrak{Q},\mathfrak{S},S)$--Sheffer sequence $\{s_n\}_{n\geq 0}$ can be calculated recursively by $s_0=1/S(1)$ and \begin{equation}\label{Eq. snQ-1}
	s_n=n\mathfrak{Q}^{-1}_{x_0}(s_{n-1})-\frac{n}{s_0}S\left(\mathfrak{Q}^{-1}_{x_0}(s_{n-1})\right),\quad n\geq1.
	\end{equation}
\end{teor}

\begin{proof} By writing $\mathfrak{Q}=\d\circ \mathfrak{P}$, where $\mathfrak{P}\in\Sigma$ is invertible and $\mathfrak{P}^{-1}=\d/\overline{B}(\d)$, we find $$\textstyle \mathfrak{Q}\circ \mathfrak{Q}^{-1}_{x_0}(p)(x)=\left(\d\circ \mathfrak{P}\right)  \left(\mathfrak{P}^{-1}\left(\int_{x_0}^x p(s)ds\right)\right)=\d\left(\int_{x_0}^x p(s)ds\right)(p)=p(x),$$ as required. Now, setting $s_0'=s_0=1/S(1)$ and $s_n'(x)$ for the left-side of (\ref{Eq. snQ-1}) we have \begin{align*}
	\mathfrak{Q}(s_n')&=n\mathfrak{Q}(\mathfrak{Q}^{-1}_{x_0}(s'_{n-1}))=ns_{n-1}', \text{ and }\\ S(s_n')&=nS(\mathfrak{Q}^{-1}_{x_0}(s'_{n-1}))-\frac{n}{s_0}S\left(\mathfrak{Q}^{-1}_{x_0}(s'_{n-1})\right)S(1)=0,
	\end{align*}	 
	for all $n\geq 1$. Thus $\{s_n'\}_{n\geq0}$ is the $(\mathfrak{Q},\mathfrak{j}(S),S)$--sequence, so  $s_n=s_n'$, for all $n$.
\end{proof}

\section{The case of Appell sequences}\label{Sec:4}

The main example of Sheffer sequences are the \textit{Appell sequences} --in honor of P. E. Appell (1880) \cite{Appell}-- corresponding to \begin{equation}\label{Eq. Q=d}
\mathfrak{Q}=\d,\quad \text{ for which } B(t)=\overline{B}(t)=t,\quad  \text{ and } q_n(x)=x^n.
\end{equation} 
In this case we see $\{p_n(x)\}_{n\geq 0}$ is an Appell sequence if $p_0\neq 0$ is a constant and
\begin{equation}\label{Eq. Appell 1}
\frac{d p_n}{dx}(x)=np_{n-1}(x),\qquad n\geq 1,\end{equation}  Equivalently, there is a unique invertible $\mathfrak{L}\in\Sigma$ satisfying $\mathfrak{L}(p_n)=x^n$, for all $n\geq0$, the sequence has a exponential series of the form ${e^{xt}}/{L(e^{xt})}$, where $L=\mathfrak{j}^{-1}(\mathfrak{L})\in\C[x]^\ast$, or they satisfy the corresponding equations to (\ref{Eq. Binom sn}).  We can refer to $\{p_n\}_{n\geq0}$ as the \textit{$(\mathfrak{L},L)$--Appell} sequence. In this setting, theorems \ref{Thm2} and \ref{Coro Last} take the following form.

\begin{teor}\label{Thm New} The $(\mathfrak{L},L)$--Appell sequence $\{p_n\}_{n\geq0}$ is the unique Appell sequence satisfying \begin{equation}\label{Eq. LTpn}
	L(T_{x_0}(p_n))=x_0^n,\quad \text{ for all } n\geq0 \text{ and } x_0\in\C.\end{equation} Furthermore, it can be calculated recursively by $p_0=1/L(1)$ and  \begin{equation}\label{Eq. pn int}
	p_n(x)=n\int_{x_0}^x p_{n-1}(s)ds-\frac{n}{p_0}L\left(\int_{x_0}^x p_{n-1}(s)ds\right).
	\end{equation} %In particular, the coefficient $p_n(0)$ has the representation $$p_n(0)=-\frac{n}{p_0}L\left(\int_{0}^x p_{n-1}(s)ds\right).$$
\end{teor}

Although we have deduced the previous theorem from the more general case of Sheffer sequences, it can obtained by directly means. We also remark that (\ref{Eq. LTpn}) is referred as the \textit{mean value property} for Appell sequences connected to random variables, see \cite[Proposition 2.7]{Ta2015} and the references therein. 

Additionally, we can also express an Appell sequence in terms of a delta operator as follows. 

\begin{prop}\label{Prop. 2} Let $\{p_n\}_{n\geq0}$ be the $(\mathfrak{L},L)$--Appell sequence with generating series $e^{xt}/L(e^{xt})=C(t)e^{xt}$. If $\mathfrak{Q}$ is a delta operator with $B(t)\in\C[[t]]$ as in equation (\ref{Eq. B(t)}), and $(C\circ B)(t)=\sum_{k=0}^\infty \frac{\a_k}{k!} t^k$,  then $$\mathfrak{L}^{-1}=\sum_{k=0}^\infty \frac{\a_k}{k!}\mathfrak{Q}^k,\quad \text{ and thus}\quad  p_n(x)=\mathfrak{L}^{-1}(x^n)=\sum_{k=0}^n \frac{\a_k}{k!} \mathfrak{Q}^k(x^n).$$
\end{prop}

\begin{proof} The operator $\mathfrak{Q}_1=\sum_{k=0}^\infty \frac{\a_k}{k!} \mathfrak{Q}^k\in \Sigma$ is invertible since $\a_0=C(0)\neq0$. Recalling that $\mathfrak{Q}(e^{xt})=\overline{B}(t)e^{xt}$, we find $$\sum_{n=0}^\infty \mathfrak{Q}_1(x^n) \frac{t^n}{n!}=\mathfrak{Q}_1(e^{xt})=(C\circ B)(\overline{B}(t))e^{xt}=C(t)e^{xt}=\sum_{n=0}^\infty \frac{p_n(x)}{n!}t^n.$$ Therefore,	 $\mathfrak{L}^{-1}(x^n)=p_n(x)=\mathfrak{Q}_1(x^n)$, for all $n$, and $\mathfrak{Q}_1=\mathfrak{L}^{-1}$ as required. Finally,  $\mathfrak{Q}_1(x^n)=\sum_{k=0}^n \frac{\a_k}{k!} \mathfrak{Q}^k(x^n),$ since $\mathfrak{Q}^k(x^n)=0$ if $k>n$ as $\mathfrak{Q}^k$ lowers the degree of a polynomial by $k$.
\end{proof}

\begin{nota}
	The previous proposition was recently studied in \cite{AdellLekuona} for $\mathfrak{Q}=\Delta_1=\Delta$, the difference operator of step one. Let us recall that for each $h\in\C^\ast$, the \textit{difference operator} $$\Delta_h:=T_h-1,\quad  \text{ i.e., }\quad  \Delta_h(p)(x)=p(x+h)-p(x),$$ constitute a delta operator  for which $\overline{B}(t)=\Delta_h(e^{xt})(0,t)=e^{xt}(e^{ht}-1)|_{x=0}=e^{ht}-1$ and $B(t)=\log(1+t)/h$. Thus $e^{xB(t)}=(1+t)^{x/h}=\sum_{n=0}^\infty (x/h)_n t^n/n!$ and $\Delta_h$ has $q_n(x)=(x/h)_n$ as basic sequence. Here $(a)_n:=a(a-1)\cdots(a-n+1)$ is the falling factorial.
\end{nota}

\begin{nota}\label{Rmk Stirling} Any pair of series $B(t), \overline{B}(t)\in t\C[[t]]$, $B'(0)\neq 0$, $\overline{B}'(0)\neq 0$, compositional inverses one of each other, define two families of numbers $\{s_B(n,k)\}_{n\geq k}$ and $\{S_B(n,k)\}_{n\geq k}$ determined by \begin{equation}\label{Eq. B Stirling}
	\frac{B(t)^k}{k!}=\sum_{n=k}^\infty s_B(n,k)\frac{t^n}{n!},\quad \frac{\overline{B}(t)^k}{k!}=\sum_{n=k}^\infty S_B(n,k)\frac{t^n}{n!},
	\end{equation} just as $B(t)=\log(1+t)$, $\overline{B}(t)=e^t-1$ define the \textit{Stirling numbers} of first and second kind \cite[p. 50]{Comtet}, \begin{equation}\label{Eq. Stirling numbers}
	\frac{\log(1+t)^k}{k!}=\sum_{n=k}^\infty s(n,k)\frac{t^n}{n!},\quad \frac{(e^t-1)^k}{k!}=\sum_{n=k}^\infty S(n,k)\frac{t^n}{n!}.
	\end{equation} If $B,\overline{B}$ are associated to the delta operator $\mathfrak{Q}$ as in (\ref{Eq. B(t)}), then (\ref{Eq. B Stirling}) induces the inverse relations \begin{equation*}
	q_n(x)=\sum_{k=0}^n s_B(n,k) x^k, \quad  x^n=\sum_{k=0}^n S_B(n,k) q_k(x),
	\end{equation*} and \begin{equation*}
	\a_k=\sum_{k=0}^n s_B(n,k) p_k(0),\quad p_k(0)=\sum_{k=0}^n S_B(n,k) \a_k,\end{equation*} where $\{\a_k\}_{k\geq0}$ and $\{p_k(0)\}_{k\geq0}$ are as in Proposition \ref{Prop. 2}.
	These follows from expanding  $e^{xB(t)}$, 
	$e^{xt}=e^{xB(\overline{B}(t))}$, $C(t)$, and $C(t)=(C\circ B)(\overline{B}(t))$, respectively, as in the usual case of Stirling numbers \cite[p. 144]{Comtet}.
\end{nota}

\section{Examples}\label{Sec:5}

This final section is aimed to apply the previous results to concrete examples focusing on the role of the functional involved. In particular, we find integral representations for Bernoulli and Euler polynomials, and also for Hermite $d$-orthogonal polynomials \cite{Douak96}. Finally, we collect in Table \ref{Table} a list of important $(\mathfrak{L},L)$--Appell sequences and their characterization via Theorem \ref{Thm New}, including the recent results of Kummer hypergeometric polynomials given in \cite{Drissi}. 

\begin{nota} Any functional $S\in\C[x]^\ast$  admits a representation of the form $$S(p)=\int_{0}^{+\infty} p(s)d\beta(s),\qquad  \mathfrak{S}(p)(x)=\int_{0}^{+\infty} p(x+s)d\beta(s),$$ for some function $\beta:(0,+\infty)\to\C$ of bounded variation as it was proved by Boas \cite{Boas} in relation to the Stieljes moment problem. Thus the characterization of Theorem \ref{Thm2} can be written as $$\int_{0}^{+\infty} s_n(x_0+s)d\beta(s)=q_n(x_0),$$ and equation (\ref{Eq. SQm}) takes the form $\int_{0}^{+\infty} \mathfrak{Q}^{(m)}(s_n)(s)d\beta(s)=n!\,\delta_{n,m}$. This reasoning contains the early characterization of  Appell sequences of Thorne \cite{Thorne}, soon after generalized by Sheffer \cite{Sheffer1945}.
\end{nota}

\begin{eje}\label{Example Traslations} Let $\{p_n(x)\}_{n\geq 0}$ be the $(\mathfrak{L},L)$--Appell sequence, $C(t)=1/L(e^{xt})$, and $\a,\beta\in\C$ with $\beta\neq0$. Then $\{p_n(x-\a)\}_{n\geq 0}$ is the $(\mathfrak{L}\circ T_\a,L\circ T_\a)$--Appell sequence and $\{\beta^{-n} p_n(\beta x)\}_{n\geq 0}$ is the $(\mathfrak{L}\circ \mathcal{H}_\a,L\circ \mathcal{H}_\a)$--Appell sequences, where $\mathcal{H}_\beta:\C[x]\to \C[x]$ is the homothecy  $\mathcal{H}_\beta(p)(x)=p(x/\beta)$. Indeed, these sequences have as generating series  $C(t)e^{(x-\a)t}=e^{xt}/L(e^{(x+\a)t})$, and  $C(t/\beta)e^{xt}=e^{xt}/L(e^{xt/\beta})$, respectively.
\end{eje}

\begin{eje}[Bernoulli polynomials] They are defined by the expansion $$\sum_{n=0}^{\infty} B_n(x)\frac{t^n}{n!}=\frac{t}{e^t-1}e^{xt},\qquad B_n(x)=\sum_{j=0}^n \binom{n}{j} B_j(0)x^{n-j}.$$ The $B_j=B_j(0)$ are the  Bernoulli numbers that satisfy $B_0=1$, $B_1=-\frac{1}{2}$ and $B_{2j+1}=0$, $j\geq 1$. We see the Bernoulli polynomials conform the Appell sequence relative to $$I(p)=\int_0^1 p(s)ds\quad \text{ and }\quad  \mathfrak{I}(p)(x):=\int_0^{1} p(x+s)ds,\quad \text{ since } I(e^{xt})=\frac{e^t-1}{t}.$$ Theorem \ref{Thm New} asserts they are characterized by the condition $\mathfrak{I}(B_n)(x)=\int_0^1 B_n(x+s)ds=x^n$, or equivalently after differentiation, by the equation $\Delta(B_{n})(x)=B_{n}(x+1)-B_{n}(x)=nx^{n-1},$ which is a well-known result. Furthermore, we can compute them recursively by $B_0(x)=1$ and $$B_n(x)=n\int_0^x B_{n-1}(t)dt-n\int_0^1\int_0^u B_{n-1}(t)dtdu.$$ 
	
	Following Remark \ref{Nota kfold} we find the $k$-fold iteration of ${I}$ is $I^{ k}(p)=\int_{[0,1]^k} p(s_1+\cdots+s_k)d\boldsymbol{s}$ which produces the \textit{$k$th order Bernoulli polynomials} $B_n^{(k)}(x)$ having  ${t^ke^{xt}}/{(e^t-1)^k}$ as generating exponential series. Moreover, we can write these polynomials in terms of $\Delta$ as $$B_n(x)=\sum_{j=0}^n \frac{(-1)^j}{j+1} \Delta^j(x^n),\quad \text{ and }\quad  B_n^{(k)}(x)=\sum_{j=0}^n \frac{k!}{(k+j)!} s(k+j,k)\Delta^j(x^n),$$ by using Proposition \ref{Prop. 2}. In fact, in this case $B(t)=\log(1+t)$ and the previous formulas follow from (\ref{Eq. Stirling numbers}) since  $((e^t-1)/t)^k\circ B(t)=\log(1+t)^k/t^k$, c.f., \cite[Theorem 5]{AdellLekuona}.
\end{eje}

An interesting question is to determine an analytic representation for the inverse operator of $\mathfrak{I}$ which in turn gives left-inverses for $\Delta$ and an analytic representation of Bernoulli polynomials. We remark that the formula below to invert $\Delta$ is familiar in the theory of difference equations and it has been used to justify Ramanujan summation, see \cite[Theorem 1]{Candel}.

\begin{prop}\label{Prop. Inv Ber} The map  $L(p)=p(0)-\frac{p'(0)}{2}-i\int_0^{+\infty} \frac{p'(is)-p'(-is)}{e^{2\pi s}-1} ds$ verifies $L(e^{xt})=t/(e^{t}-1)$. Therefore, the inverse operator of   $\,\mathfrak{I}(p)(x)=\int_0^1 p(x+s)ds$ is $$\mathfrak{I} ^{-1}(p)(x)=p(x)-\frac{p'(x)}{2}-i\int_0^{+\infty} \frac{p'(x+is)-p'(x-is)}{e^{2\pi s}-1} ds.$$ Moreover, the difference operator $\Delta$  admits the left-inverses $$\Delta^{-1}_{x_0}(p)=\int_{x_0}^x p(s)ds-\frac{p(x)}{2}-i\int_0^{+\infty} \frac{p(x+is)-p(x-is)}{e^{2\pi s}-1}ds.$$ Furthermore, the Bernoulli polynomials admit the integral representation 
	$$B_n(x)=x^n-\frac{n}{2}x^{n-1}-in\int_0^{+\infty} \frac{(x+is)^{n-1}-(x-is)^{n-1}}{e^{2\pi s}-1} ds.$$
\end{prop}

\begin{proof} For the first statement note $L(1)=1=B_0$, $L(x)=-1/2=B_1$ and $L(x^{2j+1})=0$, $j\geq1$ since the derivative of $x^{2j+1}$ is an even function. For the even powers we find $$L(x^{2j})=2j(-1)^{j+1} \int_0^{+\infty} \frac{2s^{2j-1}}{e^{2\pi s}-1} ds=B_{2j},$$ values that are familiar in the study of Abel--Plana formula \cite[p. 291]{Olver}, \cite[24.7.2]{OlverLozier}. Now, the operator $\mathfrak{L}=\mathfrak{j}(L)$ is the inverse of $\mathfrak{I}$ since $\mathfrak{I}(e^{xt})=(e^t-1)/t$. Finally, Proposition \ref{Coro Last} shows  $\Delta^{-1}_{x_0}(p)=\frac{\d}{e^{h\d}-1}\left(\int_{x_0}^x p(s)ds\right)$ and the previous example proves $B_n(x)=\mathfrak{I}^{-1}(x^n)$ as required.
\end{proof}

\begin{eje}[Euler polynomials] The Apostol--Euler polynomials are determined by $$\frac{e^{xt}}{1+\beta(e^t-1)}=\sum_{n=0}^\infty E_n(\beta;x)\frac{t^n}{n!},\quad \text{ for a fixed } \beta\neq 0.$$ They are the Appell sequence relative to $$L(p)=(1-\beta)p(0)+\beta p(1),\quad  \mathfrak{L}(p)(x)=(1-\beta)p(x)+\beta p(x+1),\quad \text{ since } L(e^{xt})=1+\beta(e^{t}-1).$$ The case $\beta=1/2$ recovers the classical Euler polynomials  $E_n(1/2;x)=E_n(x)$. Theorem \ref{Thm New} shows the Apostol--Euler are 
	characterized by $(1-\beta)E_n(\beta;x)+\beta E_n(\beta;x+1)=x^n$. Moreover,they are given recursively by $E_0(\beta;x)=1$ and $$E_n(\beta;x)=n\int_0^x E_{n-1}(\beta;t)dt-\beta n \int_0^1\int_0^u E_{n-1}(\beta;t)dtdu.$$ 
	
	The \textit{$k$th Apostol--Euler polynomials} $E_n^{(k)}(\beta;x)$ are the Appell sequence relative to the functional ${L}^{k}(p)=\sum_{j=0}^k \binom{k}{j} \beta^j (1-\beta)^{k-j} p(j)$ and with exponential generating series ${e^{xt}}/{\left(1+\beta (e^t-1)\right)^k}$. The case $\beta=2$ corresponds to the \textit{$k$th Euler polynomials} $E_n^{(k)}(x)$.
	Finally, Proposition \ref{Prop. 2} proves that $$E_n^{(k)}(\beta;x)=\sum_{j=0}^n \binom{j+k-1}{j}(-1)^j \beta^j\Delta^j(x^n),\quad \text{ and } E_n(x)=\sum_{j=0}^n \frac{(-1)^j}{2^j} \Delta^j(x^n),$$ since $(1+\beta (e^t-1))^{-k}\circ \log(1+t)=(1+\beta t)^{-k}=\sum_{j=0}^\infty \binom{j+k-1}{j} (-1)^j \beta^j t^j$, c.f., \cite[Theorem 7]{AdellLekuona}.
\end{eje}

In analogy with Proposition \ref{Prop. Inv Ber}, we can write an analytic expression for the inverse of the operator inducing the Euler polynomials. More specifically, we have.

\begin{prop}\label{Prop Inv Euler} The functional  $L(p)=\int_0^{+\infty} \frac{p\left(-\frac{1}{2}+\frac{is}{2}\right)+p\left(-\frac{1}{2}-\frac{is}{2}\right)}{e^{\pi s/2}+e^{-\pi s/2}} ds$ satisfies $L(e^{xt})=2/(e^{t}+1)$. In consequence,  $$\mathfrak{J}^{-1}(p)(x)=\int_0^{+\infty} \frac{p\left(x-\frac{1}{2}+\frac{is}{2}\right)+p\left(x-\frac{1}{2}-\frac{is}{2}\right)}{e^{\pi s/2}+e^{-\pi s/2}} ds$$ is the inverse of $\,\mathfrak{J}(p)(x)=(p(x+1)+p(x))/2$. Furthermore, the Euler polynomials admit the integral representation $$E_n(x)=\int_0^{+\infty} \frac{\left(x-\frac{1}{2}+\frac{is}{2}\right)^n+\left(x-\frac{1}{2}-\frac{is}{2}\right)^n}{e^{\pi s/2}+e^{-\pi s/2}} ds.$$
\end{prop}

\begin{proof} 	It is sufficient to establish the first formula, the remaining ones follow as in the previous proposition. For this purpose we write the Euler polynomials in terms of the \textit{Euler numbers} $E_n$ as 
	$$E_n(x)=\sum_{k=0}^n \binom{n}{k} \frac{i^k}{2^k} E_k \left(x-\frac{1}{2}\right)^{n-k}\text{ where }\quad   \frac{2}{e^t+e^{-t}}=\sum_{n=0}^\infty i^n E_{n}\frac{t^n}{n!}.$$ In fact,  $\sum_{n=0}^\infty E_n(x)\frac{t^n}{n!}={2e^{xt}}/(e^t+1)={2e^{(x-1/2)t}}/(e^{t/2}+e^{-t/2})$. Also note that $E_{2j+1}=0$, $j\geq0$. Now, by Example \ref{Example Traslations}, it is enough to show that the operator $$L'(p)=\int_0^{+\infty} \frac{p\left(\frac{is}{2}\right)+p\left(-\frac{is}{2}\right)}{e^{\pi s/2}+e^{-\pi s/2}} ds,$$ satisfies $L'(e^{xt})=2/(e^{t/2}+e^{-t/2})$. In fact, $L=L'\circ T_{-1/2}$ and therefore $L(e^{xt})=e^{-t/2}L'(e^{xt})=2/(e^t+1)$ as required. It is clear that $L'(x^{2j+1})=0$ since these are odd functions. For the even powers we also find \cite[24.7.6]{OlverLozier} $$L'(x^{2j})=\frac{(-1)^j}{2^{2j-1}}\int_0^{+\infty} \frac{s^{2j}}{e^{\pi s/2}+e^{-\pi s/2}} ds=(-1)^j\frac{E_{2j}}{2^{2j}},$$ as required.
\end{proof}

\begin{nota} Although the integral representations of Euler numbers  $$E_{2j}=2\int_{0}^{+\infty} \frac{s^{2j}}{e^{\pi s/2}+e^{-\pi s/2}}ds=\left(\frac{2}{\pi}\right)^{2j+1}\int_0^{+\infty} \frac{\ln(u)^{2j}}{u^2+1}du,$$ are known (here $u=e^{\pi s/2}$), we find instructive to include a simple proof using calculus of residues. Indeed, we can find $A_n=(2/\pi)^{n+1}\int_0^{+\infty}\frac{\ln(u)^n}{1+u^2}du$, recursively: using the branch of the logarithm $\log(z)=\ln|z|+i\text{arg}(z)$ with $-\pi/2<\text{arg}(z)<3\pi/2$, the Residue Theorem shows that $\int_{\gamma_{\epsilon,R}} \frac{\log(z)^n}{z^2+1}dz=2\pi i\,\text{Res}(\log(z)^n/(1+z^2),i)={i^n\pi^{n+1}}/{2^n}$. Here $0<\epsilon<1$, $R>1$ and $\gamma_{\epsilon,R}$ is the path formed by the segments from $-R$ to $-\epsilon$ and $\epsilon$ to $R$, and the corresponding semicircles centered at $0$ of radius $\epsilon$ and $R$, oriented positively. Letting $\epsilon\rightarrow0^+$ and $R\rightarrow+\infty$ the integral over the arcs tends to $0$ and we obtain $$\int_0^{+\infty}\frac{\ln(u)^n}{1+u^2}du+\int_0^{+\infty}\frac{(\ln(u)+i\pi)^n}{1+u^2}du=\frac{i^n\pi^{n+1}}{2^n}.$$ Thus the sequence $A_n$ satisfies  $A_n+\sum_{k=0}^{n-1} \binom{n}{k}i^k 2^k A_{n-k}=2i^n.$ If we set $A(t)=\sum_{n=0}^\infty A_n t^n/n!$, this recursion is equivalent to the equation $A(t)+A(t)e^{2it}=2e^{it}$. Consequently, $A(t)=2/(e^{it}+e^{-it})=\sum_{n=0}^\infty (-1)^n E_n t^n/n!$ and  $A_n=(-1)^n E_n$ as needed.
\end{nota}

Now we proceed to extend the integral representation and characterization of Hermite polynomials that are essentially the only Appell orthogonal sequence \cite{Shohat}. The following generalization is meant in the context of $d-$orthogonality, see \cite{Douak96} and the references therein. But first we need a remark.

\begin{nota}\label{Nota Ramification} Given $L\in\C[x]^\ast$ and an integer $m\geq1$, we can construct a functional recording only the moments of $L$ indexed by multiples of $m$. Indeed, recalling Example \ref{Example Traslations} and fixing the $m$-th root of unity $\omega_{m}:=e^{2\pi i/m}$, we see that the functional $$L_m=\frac{1}{m}(L+L\circ \mathcal{H}_{\omega_m^{-1}}+\cdots+L\circ \mathcal{H}_{\omega_m^{-(m-1)}})$$ has moments $L_m(x^{nm})=L_{nm}$ and equal to zero otherwise, i.e., $L_m(e^{xt})=\sum_{n=0}^\infty  L_{nm}\frac{t^{nm}}{(nm)!}.$ This can be checked using the identity $1+\omega_{m}^j+\cdots+\omega_{m}^{j(m-1)}=0$, valid for $j=1,\dots,m-1$.
\end{nota}

\begin{eje}[Hermite polynomials] Fix an integer $d\geq1$. We shall describe an analytic expression for the functional defining the Appell sequence determined by the expansion $$\exp\left( xt-t^{d+1}\right)=\sum_{n=0}^\infty {H}_n^{(d)}(x)\frac{t^n}{n!},$$ which correspond to a particular case of Gould-Hopper polynomials \cite{Gould}. Our approach is based on Ecalle's accelerator operators familiar in the theory of multisummability of power series, see \cite[Chapter  11]{Balser}. To this end, we recall the \textit{accelerator function}  \begin{equation}\label{Eq. Ca}
	C_\a(z):=\frac{1}{\pi}\sum_{n=0}^\infty \sin\left(\frac{(n+1)\pi}{\beta}\right)\Gamma\left(\frac{n+1}{\a}\right) \frac{z^n}{n!},\quad \text{ where } \a>1,  \frac{1}{\a}+\frac{1}{\beta}=1,
	\end{equation} and $\Gamma$ is the Gamma function. The map $C_\a$ is entire and satisfies $|C_\a(z)|\leq c_1\exp(-c_2|z|^\beta)$ on each sector $|\text{arg}(z)|\leq \theta/2<\pi/(2\beta)$, for certain constants $c_j=c_j(\a,\theta)>0$, $j=1,2$. Then, given $k'>k>0$, the \textit{acceleator operator} of index $(k',k)$,  $\mathcal{A}_{k',k}(p)(z):=z^{-k} \int_0^{+\infty} p(s) C_{k'/k}((s/z)^k) ds^k$ is well-defined for all polynomials $p\in\C[z]$. Its importance relies on the fact that $\mathcal{A}_{k',k}(z^n)=\frac{\Gamma(1+n/k)}{\Gamma\left(1+{n}/{k'}\right)}z^n$, for all $n\geq0$. Choosing $k=1$, $k'=d+1$, and $z=1$, we find that  
	$$\mathcal{A}_{d+1,1}(p):=\int_0^{+\infty} p(s) C_{d+1}(s)ds\quad \text{ has moments }\quad \mathcal{A}_{d+1,1}(x^n)=\frac{n!}{\Gamma\left(1+\frac{n}{d+1}\right)}.$$ Therefore, if $\omega_{d+1}=e^{2\pi i/(d+1)}$, Remark  \ref{Nota Ramification} proves that the functional \begin{equation}\label{Eq. Ad+1}
	\mathcal{A}^{d+1}(p):=\frac{1}{d+1}\int_0^{+\infty} \left(p(s)+p(\omega_{d+1}s)+\cdots+p(\omega_{d+1}^{d}s)\right) C_{d+1}(s)ds
	\end{equation}  satisfies  $\mathcal{A}^{d+1}(e^{xt})=\exp(t^{d+1}).$ Moreover, for $\lambda\neq0$, Example \ref{Example Traslations} shows that $\mathcal{A}^{d+1}\circ \mathcal{H}_\lambda$ produces the sequence $\{\lambda^{-n} H_n^{(d)}(\lambda x)\}$ having $\exp(xt-(t/\lambda)^{d+1})$ as generating series. In particular, if we choose $\lambda^{d+1}=-1$, say $\lambda=e^{i\pi/(d+1)}$, the generating series would be $\exp(xt+t^{d+1})$. These considerations establish the following.
	
	\begin{prop}\label{Prop. Hnd} The sequence $\{H_n^{(d)(x)}\}_{n\geq0}$ having $\exp(xt-t^{d+1})$ as exponential generating series is characterized as the Appell sequence relative to $\mathcal{A}^{d+1}$ given by (\ref{Eq. Ad+1}) and satisfying \begin{equation}\label{Eq. Hnd}
		\int_0^{+\infty} \sum_{j=0}^d {H}_n^{(d)}\left(x+e^{\frac{2\pi ij}{d+1}}s\right)C_{d+1}(s)ds=(d+1)x^n,\end{equation} where $C_{d+1}$ is Ecalle's accelerator function (\ref{Eq. Ca}). Moreover, $H_n^{(d)}$ admits the integral representation \begin{equation}\label{Eq. Hnd inverse}
		H_n^{(d)}(x)=\frac{1}{d+1} \int_0^{+\infty} \sum_{j=0}^d \left(x+e^{\frac{2\pi i(j-1/2)}{d+1}}s\right)^n C_{d+1}(s)ds.
		\end{equation}	
	\end{prop}
	
\
	
On the other hand, if in the previous paragraph we take $\lambda=\lambda_d:=(d!(d+1)^2)^{1/(d+1)}$ we recover the family $\hat{H}_n(x;d):=\lambda_d^{-n}H_n^{(d)}(\lambda_d x)$ of \textit{Hermite-type-$d$-orthogonal} polynomials, corresponding to the notion of $d-$orthogonality  \cite{Douak96}. In the case $d=1$ and $\a=\beta=2$ in (\ref{Eq. Ca}), after a direct calculation using the values $\Gamma(n+1/2)=(2n)!/(4^nn!)\sqrt{\pi}$, $n\geq0$, we find the familiar function $C_2(z)=\exp(-z^2/4)/\sqrt{\pi}$. Then, taking $\lambda=\sqrt{2}$ we find the functional $$(\mathcal{A}^2\circ \mathcal{H}_{\sqrt{2}})(p)=\frac{1}{2}\int_0^{+\infty} \left(p(s/\sqrt{2})+p(-s/\sqrt{2})\right) C_2(s)ds=\frac{1}{\sqrt{2\pi}} \int_{-\infty}^{+\infty} p(s) e^{-s^2/2}ds,$$ corresponding to the Weierstrass operator \cite[p. 746]{Drissi} who induces the classical Hermite polynomials $\{H\!e_n(x)\}_{n\in\N}$ with generating series $\exp(xt-{t^2}/{2})$. Finally, equations (\ref{Eq. Hnd}) and (\ref{Eq. Hnd inverse}) take the familiar form \cite[p. 254]{Magnuns}  $$x^n=\frac{1}{\sqrt{2\pi}}\int_{-\infty}^{+\infty} H\!e_n(x+s) e^{-s^2/2}ds,\qquad H\!e_n(x)=\frac{1}{\sqrt{2\pi}} \int_{-\infty}^{+\infty} (x+is)^n e^{-s^2/2}ds.$$ 
\end{eje}

We conclude with one final worked example in relation with functionals induced by entire functions.

\begin{eje} Given an entire function $F(z)=\sum_{n=0}^\infty f_n z^n$ and $\lambda\in\C$  such that $F(\lambda)\neq 0$, let $$L_{F,\lambda}(p)=\frac{1}{F(\lambda)} \sum_{k=0}^\infty p(k) f_k \lambda^k,\quad\text{ and }\quad  \mathfrak{L}_{F,\lambda}(p)(x)=\frac{1}{F(\lambda)} \sum_{k=0}^\infty p(x+k) f_k \lambda^k.$$ To see these are well-defined we use the differential operator $\delta=z\frac{\d}{\d z}$ to note that $\delta^j(F)(z)=\sum_{k=1}^\infty k^j f_k z^k$, $j\geq1$ are again entire and thus  $${L}_{F,\lambda}(x^n)= \frac{\delta^n(F)(\lambda)}{F(\lambda)},\quad \text{ and }\quad \mathfrak{L}_{F,\lambda}(x^n)(x)=\sum_{j=0}^n \binom{n}{j} \frac{\delta^j(F)(\lambda)}{F(\lambda)} x^{n-j}.$$ Moreover, $L_{F,\lambda}(e^{xt})=\sum_{k=0}^\infty e^{kt} f_k\lambda^k/{F(\lambda)}={F(\lambda e^{t})}/{F(\lambda)}$. In this way we obtain an Appell sequence $\{P_{F,\lambda,n}(x)\}_{n\in\N}$ characterized by  the equation $$\sum_{k=0}^\infty P_{F,\lambda,n}(x+k) f_k \lambda^k=F(\lambda)x^n.$$ If $F$ has no zeros in the complex plane, $1/F(z)=\sum_{k=0}^\infty f_k' z^k$ is again entire and $L_{1/F,\lambda}(e^{xt})=F(\lambda)/F(\lambda e^t)$. Therefore, $\mathfrak{L}_{F,\lambda}^{-1}=\mathfrak{L}_{1/F,\lambda}$ and we can invert the previous equation to obtain \begin{equation}\label{Eq. P F lambda n}
	P_{F,\lambda,n}(x)=F(\lambda)\sum_{k=0}^\infty (x+k)^n f_k'\lambda^k=\sum_{j=0}^n \binom{n}{j} \frac{\delta^j(1/F)(\lambda)}{(1/F)(\lambda)}x^{n-j}.
	\end{equation} 
	
	Examples of this situation are given by $F(z)=\exp(P(z))$, where $P$ is a non-constant polynomial as considered by Touchard \cite{Touchard}. The best-known case corresponds to $F(z)=e^z$, for which $$L_{e^z,\lambda}(e^{xt})=\exp(\lambda(e^t-1))=\sum_{k=0}^\infty T_k(\lambda)\frac{t^k}{k!},\quad \text{ where } T_k(\lambda)=e^{-\lambda}\cdot{\delta^n(e^z)(\lambda)}=\sum_{k=0}^n S(n,k) \frac{\lambda^k}{k!}$$ are the \textit{exponential} or \textit{Touchard} polynomials \cite[p. 63]{Roman} --recall equation (\ref{Eq. Stirling numbers})--. In particular, (\ref{Eq. P F lambda n}) takes the form $P_{e^z,\lambda,n}(x)=\sum_{j=0}^n \binom{n}{j} \sum_{l=0}^j S(j,l) {(-\lambda)^l}x^{n-j}/{l!}$.  \cite[Example 4.4]{Borwein}
\end{eje}

Table \ref{Table} contains more examples of Appell sequences as N\"{o}rlund \cite{Norlund}, Laguerre \cite[p. 108]{Roman}, Strodt \cite{Borwein}, and the Bernoulli-type polynomials  \cite{Tempesta} (see Lemma 1 where $a$ should be only equal to $1$). Moreover, the Bernoulli hypergeometric \cite{Hassen} and Kummer hypergeometric Bernoulli polynomials \cite{Drissi}.

\renewcommand{\arraystretch}{1.8}
\begin{table}[h]
	\caption{Some families of Appell polynomials}
	\label{Table}

	\begin{tabular}{ccc}		
		\hline
		\multicolumn{1}{|c|}{Polynomials} & \multicolumn{1}{c|}{Functional $L(p)$}                  & \multicolumn{1}{c|}{Indicator series $L(e^{xt})$ /  Characterization}                  \\ \hline
		
		\multicolumn{1}{|c|}{\makecell{Monomials\\ $(x-a)^n$}} & \multicolumn{1}{c|}{\multirow{2}{*}{\makecell{$p(a)$\\$\empty$}}} & \multicolumn{1}{c|}{\multirow{2}{*}{\makecell{$e^{at}$\\$\empty$}}} \\ \hline

		\multicolumn{1}{|c|}{\multirow{2}{*}{\makecell{Bernoulli \\ $B_n(x)$}}} & \multicolumn{1}{c|}{\multirow{2}{*}{ $\displaystyle\int_0^1 p(t)dt$ }} & \multicolumn{1}{c|}{ ${(e^t-1)}/{t}$ } \\ \cline{3-3} 
		\multicolumn{1}{|c|}{}                  & \multicolumn{1}{c|}{}                  & \multicolumn{1}{c|}{ $B_{n+1}(x+1)-B_{n+1}(x)=(n+1)x^n$ } \\ \hline

	\multicolumn{1}{|c|}{\multirow{2}{*}{\makecell{$k$th Bernoulli \\ $B_n^{(k)}(x)$}}} & \multicolumn{1}{c|}{\multirow{2}{*}{ $\displaystyle \int_{[0,1]^k} p(s_1+\cdots+
			s_k)d\boldsymbol{s}$ }} & \multicolumn{1}{c|}{ $\displaystyle (e^{t}-1)^k/t^k$ } \\ \cline{3-3} 
\multicolumn{1}{|c|}{}                  & \multicolumn{1}{c|}{}                  & \multicolumn{1}{c|}{ $\int_{[0,1]^k} B_n^{(k)}(x+s_1+\cdots+s_k)d\boldsymbol{s}=x^n$ } \\ \hline

		\multicolumn{1}{|c|}{\makecell{Bernoulli\\N\"{o}rlund}} & \multicolumn{1}{c|}{ $\quad \displaystyle\int_{[0,1]^k} p(\omega_1s_1+\cdots+\omega_k 
			s_k)d\boldsymbol{s}\quad $ } & \multicolumn{1}{c|}{$ \prod_{j=1}^k \frac{e^{\omega_j t}-1}{\omega_j t}$} \\ \cline{2-3} 
		\multicolumn{1}{|c|}{$\quad B_n^{(k)}(x|\omega)\quad $} & \multicolumn{1}{c|}{$\omega=(\omega_1,\dots,\omega_k)\in\C^k$} & \multicolumn{1}{c|}{$\quad \int_{[0,1]^k} B_n^{(k)}(x+\sum_{j=1}^k \omega_js_j\,|\,\omega)d\boldsymbol{s}=x^n\quad $} \\ \hline

		\multicolumn{1}{|c|}{\multirow{2}{*}{\makecell{Apostol\\Euler \\ $E_n(\beta;x)$}}} & \multicolumn{1}{c|}{\multirow{2}{*}{ $\displaystyle (1-\beta)p(0)+\beta p(1)$ }} & \multicolumn{1}{c|}{ $\displaystyle 1+\beta(e^t-1)$ } \\ \cline{3-3} 
		\multicolumn{1}{|c|}{}                  & \multicolumn{1}{c|}{}                  & \multicolumn{1}{c|}{ $(1-\beta)E_n(\beta;x)+\beta E_n(\beta;x+1)=x^n$ } \\ \hline

		\multicolumn{1}{|c|}{\multirow{2}{*}{\makecell{Euler \\ $E_n(x)$}}} & \multicolumn{1}{c|}{\multirow{2}{*}{ $\displaystyle \frac{p(0)+p(1)}{2}$ }} & \multicolumn{1}{c|}{ $\displaystyle {(1+e^t)}/{2}$ } \\ \cline{3-3} 
		\multicolumn{1}{|c|}{}                  & \multicolumn{1}{c|}{}                  & \multicolumn{1}{c|}{ $E_n(x)+E_n(x+1)=2x^n$ } \\ \hline

			\end{tabular}
	\end{table}

\renewcommand{\arraystretch}{1.8}
\begin{table}[]
%\caption{Some families of Appell polynomials}
%\label{Table}

\begin{tabular}{ccc}	

\hline
	
		\multicolumn{1}{|c|}{ \makecell{$k$th Apostol\\Euler} } & \multicolumn{1}{c|}{\multirow{2}{*}{ $\displaystyle\sum_{j=0}^k \binom{k}{j} \beta^j (1-\beta)^{k-j} p(j)$ } } & \multicolumn{1}{c|}{ $\displaystyle \left(1+\beta (e^t-1)\right)^{k}$} \\ \cline{3-3} 
	\multicolumn{1}{|c|}{ $E_n^{(k)}(\beta;x)$ } & \multicolumn{1}{c|}{}                  & \multicolumn{1}{c|}{ $\sum_{j=0}^k \binom{k}{j} \beta^j (1-\beta)^{k-j} E_n^{(k)}(\beta;x+j)=x^n$ }                  
	\\ \hline

	\multicolumn{1}{|c|}{\makecell{Euler\\N\"{o}rlund}} & \multicolumn{1}{c|}{ $\displaystyle 2^{-k} \sum_{n_j\in\{0,1\}}  p\Big(\sum_{j=1}^k n_j\omega_j\Big)$ } & \multicolumn{1}{c|}{ $ 2^{-k}\prod_{j=1}^k (e^{\omega_j t}+1)$} \\ \cline{2-3} 
	\multicolumn{1}{|c|}{$E_n^{(k)}(x\,|\,\omega)$} & \multicolumn{1}{c|}{$\omega=(\omega_1,\dots,\omega_k)\in\C^k$} & \multicolumn{1}{c|}{$\sum_{n_j\in\{0,1\}}  E_n^{(k)}\left(x+\sum_{j=1}^k n_j\omega_j \,|\,\omega\right)=2^{k} x^n$} \\ \hline

	\multicolumn{1}{|c|}{\multirow{2}{*}{\makecell{$w$-Strodt\\$S_{n,w}(x)$}}} & \multicolumn{1}{c|}{ $\sum_{j=1}^{N} w_j p\left(x_j\right)$ } & \multicolumn{1}{c|}{ $\sum_{j=1}^N w_j e^{x_jt}$ } \\ \cline{2-3} 
	\multicolumn{1}{|c|}{  }                  & \multicolumn{1}{c|}{ $x_j\in\R$, $0<w_j<1$,  $\sum_{j=1}^N w_j=1$ } & \multicolumn{1}{c|}{ $\sum_{j=0}^{n-1} S_{n,w}\left(x+w_j\right)=x^n$ } \\ \hline

	\multicolumn{1}{|c|}{ Bernoulli-type } & \multicolumn{1}{c|}{ $\sum_{j=l}^{m} a_j\int_0^j p(s)ds$ } & \multicolumn{1}{c|}{ $\sum_{j=l}^{m} a_je^{jt}/t$ } \\ \cline{2-3} 
	\multicolumn{1}{|c|}{ $B_{n,1}^{(m-l)}(x)$ } & \multicolumn{1}{c|}{ $l,m\in\Z,\, \sum_{j=l}^m a_j=0,\,\,  \sum_{j=l}^m ja_j=1$ } & \multicolumn{1}{c|}{ $\sum_{j=l}^{m} a_j B_{n,1}^{(m-l)}(x+j)=nx^{n-1}$ } \\ \hline

	\multicolumn{1}{|c|}{ Hermite } & \multicolumn{1}{c|}{\multirow{2}{*}{  $\displaystyle \frac{1}{\sqrt{2\pi}}\int_{-\infty}^{+\infty} p(s) {e^{-s^2/2}} ds$ } } & \multicolumn{1}{c|}{ $\displaystyle \exp(t^2/2)$ } \\ \cline{3-3} 
	\multicolumn{1}{|c|}{ $H\!e_n(x)$ } & \multicolumn{1}{c|}{}                  & \multicolumn{1}{c|}{ $\int_{-\infty}^{+\infty} H\!e_n(x+s) e^{-s^2/2}ds=\sqrt{2\pi}x^n$ }                  
	\\ \hline
	
	\multicolumn{1}{|c|}{ $d$-Hermite } & \multicolumn{1}{c|}{\multirow{2}{*}{  $\displaystyle\int\limits_0^{+\infty} \sum_{j=0}^d p\left(e^{\frac{2\pi ij}{d+1}}s\right) \frac{C_{d+1}(s)}{d+1}ds$ } } & \multicolumn{1}{c|}{ $\exp(t^{d+1})$, $d\geq0$  integer  } \\ \cline{3-3} 
	\multicolumn{1}{|c|}{ ${H}_n^{(d)}(x)$ } & \multicolumn{1}{c|}{}                  & \multicolumn{1}{c|}{ $\int\limits_0^{+\infty} \sum_{j=0}^d {H}_n^{(d)}(x+e^{\frac{2\pi ij}{d+1}}s)\frac{C_{d+1}}{d+1}(s)ds=x^n$ }                  
	\\ \hline

	\multicolumn{1}{|c|}{ Laguerre } & \multicolumn{1}{c|}{\multirow{2}{*}{  $\displaystyle \int\limits_0^{+\infty} p(s) \frac{s^\a e^{-s}}{\Gamma(1+\a)} ds$, $\text{Re}(\a)>-1$ } } & \multicolumn{1}{c|}{ $\displaystyle (1-t)^{-\a-1}$   } \\ \cline{3-3} 
	\multicolumn{1}{|c|}{ $(-1)^n n! L_n^{(\a-n)}(x)$ } & \multicolumn{1}{c|}{}                  & \multicolumn{1}{c|}{ $\int_0^{+\infty} L_n^{(\a-n)}(x+t)\frac{t^\a e^{-t}}{\Gamma(1+\a)}dt=\frac{(-1)^n x^n}{n!}$ }                  
	\\ \hline
		
		\multicolumn{1}{|c|}{ \makecell{Bernoulli  \\ hypergeometric}} & \multicolumn{1}{c|}{\multirow{2}{*}{   $\displaystyle N\int_0^1 p(s) (1-s)^{N-1}ds$, $N\geq1$} } & \multicolumn{1}{c|}{ $ (e^t-\sum_{j=0}^{N-1} t^j/j!)/(t^N/N!)$ } \\ \cline{3-3} 
		\multicolumn{1}{|c|}{ $B_{N,n}(x)$ } & \multicolumn{1}{c|}{}                  & \multicolumn{1}{c|}{ $N\int_0^1 B_{N,n}(x+s)(1-s)^{N-1}ds=x^n$ }                  
		\\ \hline

		\multicolumn{1}{|c|}{ \makecell{Kummer \\ hypergeometric} } & \multicolumn{1}{c|}{ $  \frac{\Gamma(a+b)}{\Gamma(a)\Gamma(b)}\int_0^1p(s)s^{a-1}(1-s)^{b-1}ds$ } & \multicolumn{1}{c|}{ $1+\sum_{n=1}^\infty \frac{a(a+1)\cdots(a+n-1)}{(a+b)(a+b+1)\cdots(a+b+n-1)}\frac{t^n}{n!}$ } \\ \cline{2-3} 
		\multicolumn{1}{|c|}{ $B_{a,b,n}(x)$ } & \multicolumn{1}{c|}{ $\text{Re}(a), \text{Re}(b)>0$ } & \multicolumn{1}{c|}{ $\int_0^1 B_{a,b,n}(x+s)\frac{s^{a-1}}{\Gamma(a)}\frac{(1-s)^{b-1}}{\Gamma(b)}ds=\frac{x^n}{{\Gamma(a+b)}}$ } \\ \hline

		\multicolumn{1}{|c|}{\multirow{2}{*}{$P_{F,\lambda,n}(x)$}} & \multicolumn{1}{c|}{$\displaystyle  \frac{1}{F(\lambda)} \sum_{k=0}^\infty   \frac{p(x+k)}{k!}F^{(k)}(0)\lambda^k$} & \multicolumn{1}{c|}{$\displaystyle F(\lambda e^t)/F(\lambda)$} \\ \cline{2-3} 
		\multicolumn{1}{|c|}{}                  & \multicolumn{1}{c|}{ $F:\C\to\C$ entire, $F(\lambda)\neq 0$ } & \multicolumn{1}{c|}{$\sum_{k=0}^\infty P_{F,\lambda,n}(x+k) \frac{F^{(k)}(0)\lambda^k}{k!} =F(\lambda)x^n$} \\ \hline
		
	\end{tabular}
\end{table}

\bibliographystyle{plain}
\bibliography{Carrillo_AppellSheffer}

\begin{thebibliography}{10}

\bibitem{Aceto2017}
L.~Aceto and I.~Cação.
\newblock A matrix approach to {Sheffer} polynomials.
\newblock {\em J. Math Anal. Appl.}, 446(1):87--100, 2017.

\bibitem{Aceto2015}
L.~Aceto, H.R. Malonek, and G.~Tomaz.
\newblock A unified matrix approach to the representation of appell
  polynomials.
\newblock {\em Integral Transforms Spec. Funct.}, 26(6):426--441, 2015.

\bibitem{AdellLekuona17}
J.~A. Adell and A.Lekuona.
\newblock Binomial convolution and transformations of {Appell} polynomials.
\newblock {\em J. Math. Anal. Appl.}, 456(1):16--33, 2017.

\bibitem{AdellLekuona}
J.~A. Adell and A.~Lekuona.
\newblock Closed form expressions for {Appell} polynomials.
\newblock {\em Ramanujan J.}, 49:567--583, 2019.

\bibitem{Appell}
P.E. Appell.
\newblock Sur une classe de polyn\^{o}mes.
\newblock {\em Ann. Sci. \'{E}c. Norm. Sup\'{e}r.}, 9:119--144, 1880.

\bibitem{Balser}
W.~Balser.
\newblock {\em Formal power series and linear systems of meromorphic ordinary
  differential equations}.
\newblock Universitext. Springer-Verlag, New York, 2000.

\bibitem{Boas}
R.~P. Boas.
\newblock Stieljes moment problem for functions of bounded variation.
\newblock {\em Bull. Amer. Math. Soc.}, 45:399--404, 1939.

\bibitem{Borwein}
J.~M. Borwein, N.~J. Calkin, and D.~Manna.
\newblock {Euler}-{Boole} summation revisited.
\newblock {\em Amer. Math. Monthly}, 116(5):387--412, 2009.

\bibitem{Bourbaki}
N.~Bourbaki.
\newblock {\em Fonctions d'une variable réelle. Théorie élémentaire}.
\newblock Les Éléments de mathématique. Edition originale publiée par
  Hermann, Paris, 1976. Springer-Verlag, Berlin Heidelberg, 2007.

\bibitem{Candel}
B.~Candelpergher.
\newblock {\em Ramanujan summation of divergent series}, volume 2185 of {\em
  LNM}.
\newblock Springer International Publishing, 2017.

\bibitem{Comtet}
L.~Comtet.
\newblock {\em Advanced combinatorics. The art of finite and infinite
  expansions}.
\newblock Revised and enlarged edn. D. Reidel Publishing Co., Dordrecht, 1974.

\bibitem{CosLongo2014}
F.~A. Costabile and E.J. Longo.
\newblock An algebraic approach to {Sheffer} polynomial sequences.
\newblock {\em Integral Transforms Spec. Funct.}, 25(4):295--311, 2010.

\bibitem{CosLongo2010}
F.~A. Costabile and E.J. Longo.
\newblock A determinantal approach to {Appell} polynomials.
\newblock {\em J. Comp. Applied Math.}, 236:1528--1542, 2010.

\bibitem{Douak96}
K.~Douak.
\newblock The relation of the $d$-orthogonal polynomials to the {Appell}
  polynomials.
\newblock {\em J. Comp. Applied Math.}, 70:279--295, 1996.

\bibitem{Drissi}
D.~Drissi.
\newblock Characterization of kummer hypergeometric {Bernoulli} polynomials and
  applications.
\newblock {\em C. R. Acad. Sci. Paris, Ser. I}, 357:743--751, 2019.

\bibitem{Gould}
H.~W. Gould and A.~T. Hopper.
\newblock Operational formulas connected with two generalizations of {Hermite}
  polynomials.
\newblock {\em Duke Math. J.}, 29(1):51--63, 1962.

\bibitem{Hassen}
A.~Hassen and H.~Nguyen.
\newblock Hypergeometric {Bernoulli} polynomials and {Appell} sequences.
\newblock {\em Int. J. Number Theory}, 5(4):767--774, 2008.

\bibitem{Magnuns}
W.~Magnus, F.~Oberhettinger, and R.~P. Soni.
\newblock {\em Formulas and theorems for special functions of mathematical
  physics}.
\newblock Springer-Verlag, 1966.

\bibitem{Navas2018}
L.~M. Navas, F.~J. Ruiz, and J.~L. Varona.
\newblock Appell polynomials as values of special functions.
\newblock {\em J. Math. Anal. Appl.}, 459(1):419--436, 2018.

\bibitem{Norlund}
N.~E. Nörlund.
\newblock {\em Vorlesungen über differenzen-rechnung}, volume~13 of {\em
  Grundlehren der mathematischen Wissenschaften}.
\newblock Springer-Verlag, Berlin Heidelberg, 1924.

\bibitem{Olver}
F.~W. Olver.
\newblock {\em Asymptotics and special functions}.
\newblock A. K. Peters, Wellesley, Massachusetts, 1997.

\bibitem{OlverLozier}
F.~W. Olver, D.W. Lozier, R.~F. Boisvert, and C.~W.~Clark (Eds.).
\newblock {\em NIST {Handbook} of mathematical functions}.
\newblock US Department of Commerce, Washington DC, 2010.

\bibitem{Roman}
S.~Roman.
\newblock {\em Umbral calculus}.
\newblock Academic Press, Orlando, Florida, 1984.

\bibitem{RomanRota}
S.~Roman and G.~C. Rota.
\newblock Umbral calculus.
\newblock {\em Adv. Math.}, 27(1):95--128, 1978.

\bibitem{Rota}
G.~C. Rota.
\newblock {\em Finite operator calculus}.
\newblock Academic Press, 1975.

\bibitem{Sheffer1939}
I.~M. Sheffer.
\newblock Some properties of polynomial sets of type zero.
\newblock {\em Duke Math J.}, 5(3):590--622, 1939.

\bibitem{Sheffer1945}
I.~M. Sheffer.
\newblock Note on {Appell} polynomials.
\newblock {\em Bull. Amer. Math. Soc.}, 51(10):739--744, 1945.

\bibitem{Shohat}
J.~Shohat.
\newblock The relation of the classical orthogonal polynomials to the
  polynomials of {Appell}.
\newblock {\em Am. J. Math.}, 58(3):453--464, 1936.

\bibitem{Sun2007}
P.~Sun.
\newblock Moment representation of {Bernoulli polynomial}, {Euler} polynomial
  and {Gegenbauer} polynomials.
\newblock {\em Statist. Probab. Lett.}, 77(7):748--751, 2007.

\bibitem{Ta2015}
B.~Q. Ta.
\newblock Probabilistic approach to {Appell} polynomials.
\newblock {\em Expo. Math.}, 33(3):269--294, 2015.

\bibitem{Tempesta}
P.~Tempesta.
\newblock On {Appell} sequences of polynomials of {Bernoulli} and {Euler} type.
\newblock {\em J. Math. Anal. Appl.}, 341(2):1295--1310, 2008.

\bibitem{Thorne}
C.~J. Thorne.
\newblock A property of {Appell} sets.
\newblock {\em Amer. Math. Monthly}, 52(4):191--193, 1945.

\bibitem{Touchard}
J.~Touchard.
\newblock Sur les cycles des substitutions.
\newblock {\em Acta Math.}, 70:243--297, 1939.

\bibitem{Wang2014}
W.~Wang.
\newblock A determinantal approach to {Sheffer} sequences.
\newblock {\em Linear Algebra Its Appl.}, 463:228--254, 1939.

\bibitem{Yang}
Y.~Yang.
\newblock Determinant representations of {Appell} polynomial sequences.
\newblock {\em Oper. Matrices}, 2(4):517--524, 2008.

\end{thebibliography}

%\begin{thebibliography}{3}
%\bibliography{samplebib}
%\end{thebibliography}

\end{document}